\documentclass[12pt,twoside]{article}
\usepackage[T1]{fontenc}
\usepackage{amssymb, latexsym, amsmath, fancyhdr, hyperref, graphicx}

\hypersetup{colorlinks=true}

\newtheorem{theorem}{Theorem}[section]
\newtheorem{lemma}[theorem]{Lemma}
\newtheorem{lemma*}[theorem]{Lemma}

\newtheorem{corollary}[theorem]{Corollary}

\newenvironment{proof}[1][Proof]{\begin{trivlist}
\item[\hskip \labelsep {\bfseries #1}]}{\end{trivlist}}

\newcommand{\qed}{\nobreak \ifvmode \relax \else
      \ifdim\lastskip<1.5em \hskip-\lastskip
      \hskip1.5em plus0em minus0.5em \fi \nobreak
      \vrule height0.75em width0.5em depth0.25em\fi}
      
\newcommand{\ric}{\text{Ric}}
\newcommand{\fut}{\text{Fut}}
\newcommand{\kahler}{\text{K\"ahler }}

\DeclareSymbolFont{AMSb}{U}{msb}{m}{n}
\DeclareMathSymbol{\N}{\mathbin}{AMSb}{"4E}
\DeclareMathSymbol{\Z}{\mathbin}{AMSb}{"5A}
\DeclareMathSymbol{\R}{\mathbin}{AMSb}{"52}
\DeclareMathSymbol{\Q}{\mathbin}{AMSb}{"51}
\DeclareMathSymbol{\I}{\mathbin}{AMSb}{"49}
\DeclareMathSymbol{\C}{\mathbin}{AMSb}{"43}

\begin{document}

\title{The Futaki Invariant of Kähler Blowups with Isolated Zeros via Localization}
\author{Luke Cherveny}

\maketitle

\begin{abstract}


We present an analytic proof of the relationship between the Calabi-Futaki invariant for a Kähler manifold relative to a holomorphic vector field with a nondegenerate zero and the corresponding invariant of its blowup at that zero, restricting to the case that zeros on the exceptional divisor are isolated.  This extends the results of Li and Shi \cite{lishi:2015} for Kähler surfaces.  We also clarify a hypothesis regarding the normal form of the vector field near its zero.  An algebro-geometric proof was given by Székelyhidi \cite{szekelyhidi:2015} by reducing the situation to the case of projective manifolds for rational data and using Donaldson-Futaki invariants.  Our proof will be an application of degenerate localization.
\end{abstract}

\section{Introduction}

Let $M$ be a compact Kähler manifold with Kähler metric $\omega$.  A fundamental question in K\"ahler geometry asks whether the class $[\omega]$ contains a canonical Kähler metric.  When the first Chern class $c_1(M)$ is zero, celebrated work of Yau \cite{yau:1978} established the existence of a unique Ricci-flat K\"ahler metric in every K\"ahler class, while in the case of negative first Chern class, Yau \cite{yau:1978} and Aubin \cite{aubin:1976} independently proved existence and uniqueness of a Kähler-Einstein metric in $c_1(M)$.  The existence of K\"ahler-Einstein metrics when $c_1(M) > 0$ has recently been addressed by Chen-Donaldson-Sun \cite{cds:2015}.

More generally, one could ask whether a K\"ahler class $\Omega \in H^{1,1}(M,\R)$ contains a constant scalar curvature Kähler (cscK) metric.  The question is quite subtle and conjecturally related to the algebro-geometric stability of $M$; see \cite{szekelyhidibook} for a survey and references.  One obstruction to the existence of a cscK metric is a generalization due to Calabi \cite{calabi:1985} of Futaki's famous obstruction to K\"ahler-Einstein metrics first defined in \cite{futaki:1983}.  For any K\"ahler class $\Omega$, this \emph{Calabi-Futaki invariant} is a certain character on the Lie algebra of holomorphic vector fields $\mathfrak{h}$

\[
\fut(\Omega, \cdot): \mathfrak{h} \rightarrow \C.
\]

\noindent whose vanishing on $\mathfrak{h}$ is necessary for $\Omega$ to support a cscK metric.

%
%
%
%

One approach to calculating the Calabi-Futaki invariant, at least when $M$ is algebraic, is via the algebraic Donaldson-Futaki invariant \cite{donaldson:2005}.  Another method is localization, which will be our approach in this paper.  When the zero locus of $X$ is nondegenerate, Tian \cite{tian:1996} gave a complete formula reducing $\fut(\Omega, X)$ to a calculation on $\text{Zero}(X)$.  When the zero locus of $X$ is degenerate, localization calculations are quite difficult.  See section 2 for details.

In this paper we study the following situation where degenerate localization calculations naturally arise: With $M$ as above, suppose that $p\in M$ is a zero of $X \in \mathfrak{h}$.  The blowup $\pi: \tilde M \rightarrow M$ at $p$ then admits a holomorphic lift $\tilde X$ of $X$ as well as a natural K\"ahler class

\[
\tilde \Omega = \pi^*\Omega - \epsilon [E],
\]
 
\noindent where $E$ is the exceptional divisor and $\epsilon$ is sufficiently small \cite{griffithsharris}.  A natural question is, what is the relationship between $\fut(\Omega,X)$ and $\fut(\tilde \Omega, \tilde X)$?

We will limit ourselves to the following assumption on $X$ at $p$:

\begin{quote}
($\star$) The Jordan canonical form of the linearization $DX$ at $p$ does not contain multiple Jordan blocks for the same eigenvalue.
\end{quote}

\noindent Geometrically, ($\star$) means $\tilde X$ does not contain a positive dimensional zero locus.

Our main theorem is 

\begin{theorem}\label{maintheorem} Let $\pi: \tilde M \rightarrow M$ be the blow-up of $n$-dimensional K\"ahler manifold $M$ at isolated non-degenerate zeros $\{p_1, \dots, p_k\}$ of a holomorphic vector field $X$ with zero-average holomorphy potential $\theta_X$. When ($\star$) holds, the Futaki invariant for $\tilde M$ with respect to the class $\tilde \Omega = \pi^*\Omega - \sum \epsilon_i [E_i]$ and the natural holomorphic extension $\tilde X$ of $X$ to $\tilde M$ satisfies

\[
\textrm{Fut}_{\tilde M}(\tilde \Omega, \tilde X) = \textrm{Fut}_M(\Omega, X) - \sum_{i = 1}^k n(n-1) \theta_X(p) \epsilon_i^{n-1} +  O(\epsilon^n)
\]

\end{theorem}

Li and Shi established the result for K\"ahler surfaces \cite{lishi:2015}.  In fact, they addressed when $p$ is not isolated or when the zero locus of $\tilde X$ is $E \cong \C P^1$, which are necessarily nondegenerate situations for surfaces but would be quite formidable more generally.  An algebro-geometric proof was given by Székelyhidi \cite{szekelyhidi:2015} by reducing the situation to the case of projective manifolds for rational data and using Donaldson-Futaki invariants, with related results recently appearing in the work of Dervan-Ross \cite{dervanross:2017} and Dyrefelt \cite{dyrefelt:2016}.  We give an analytic proof via localization, for which the $n=2$ case of Li and Shi is a special case.

One application of Theorem \ref{maintheorem} is to show the nonexistence of cscK metrics on certain blow-ups:

\begin{corollary}
If a K\"ahler manifold $M$ admits a constant scalar curvature metric $\omega \in \Omega$ and there exists a holomorphic vector field $X$ on $M$ vanishing at $p$ and such that $\theta_X(p) \neq 0$, then the blowup $\tilde M$ at $p$ does not admit cscK metrics in any class $\pi^*\Omega - \epsilon [E]$ for small $\epsilon$.
\end{corollary}

The organization of this paper is as follows. Section 2 provides background concerning the Futaki invariant, summarizes Tian's application of Bott localization to its calculation, and sets up the blowup calculation.  Section 3 gives a proof of Theorem \ref{maintheorem} in the special case that the linearization of $X$ at $p$ has a single Jordan block.  It relies on Lemma \ref{maxdegenlemma}, which generalizes the main calculation of \cite{lishi:2015}.  Section 4 addresses the general case of multiple Jordan blocks with distinct eigenvalues.  Section 5 discusses the normal forms of a holomorphic vector field $X$ about a singular point and establishes that for the purposes of Theorem \ref{maintheorem}, the simplifying assumption that $X$ is locally biholomorphic to its linearization as in Section 4 is sufficient.  Section 6 is an appendix that contains a proof of Lemma \ref{lemmagk} used to sum localization calculations.

\section{Background}

A reference for much of this section is \cite{tian:1996}; to align conventions, we too define the \kahler form $\omega$ and Ricci form $\ric$ without the usual $\sqrt{-1}$, and let the \kahler class be 

\[
\Omega = \left[\frac{\sqrt{-1}}{2\pi}\omega\right] \in H^{1,1}(M,\R)
\]

We first recall the definition of the Calabi-Futaki invariant.  Let $(M,\omega)$ be a compact \kahler manifold and $F$ the smooth function uniquely determined by

\[
S(\omega) - \bar S = \Delta F \qquad \qquad \int_M (e^F - 1)\omega^n = 0
\]

\noindent where $S(\omega)$ is scalar curvature, $\Delta$ is the Laplacian with respect to $\omega$, and
\[
\bar S = \frac{\int S(\omega) \Omega^n}{\int \Omega^n} = \frac{n \int c_1(M)\cup \Omega^{n-1}}{\int \Omega^n}
\]

\noindent is the average scalar curvature.

The \emph{Calabi-Futaki invariant} is defined for each K\"ahler class $\Omega$ to be a functional on the Lie algebra of holomorphic vector fields $\mathfrak{h}$

\[
\fut(\Omega, \cdot): \mathfrak{h} \rightarrow \C
\]

\noindent given by

\[
\fut(\Omega, X) = \left( \frac{\sqrt{-1}}{2\pi} \right)^n \int_M X(F) \; \omega^n
\]

Futaki \cite{futaki:1988} and Calabi \cite{calabi:1985} showed $\fut(\Omega, \cdot)$ is a Lie algebra character and that the definition is in fact independent of the choice of metric $\omega$ in its K\"ahler class $\Omega$, justifying the notation and making its vanishing for all $X \in \mathfrak{h}$ necessary for $\Omega$ to contain a cscK metric.

Following Tian \cite{tian:1996}, we now explain the localization of $\fut(\Omega,X)$.  For every $X \in \mathfrak{h}$, Hodge theory provides a harmonic $(0,1)$-form $\alpha$ and a smooth function $\theta_X$, unique up to addition of a constant, such that

\begin{equation}\label{hodge}
i_X \omega = \alpha - \bar \partial \theta_X.
\end{equation}

\noindent Equivalently, $\theta_X$ is \emph{holomorphy potential} for $X$\footnote{Or rather, recalling that $\omega$ is defined without a $\sqrt{-1}$ on it, $\sqrt{-1}\theta$ is holomorphy potential in the sense that $f: M \rightarrow \C$ is holomorphy potential for holomorphic vector field $X = g^{i\bar j} (\partial_{\bar j} f) \partial_i$, i.e. $X$ is the $(1,0)$ part of the Riemannian gradient of $f$ (up to a factor of 2).}.  By applying $\bar \partial^*$ to both sides of (\ref{hodge}) and using integration by parts,

%


\begin{align}\label{dividebyn}
\fut(\Omega, X) &= \left( \frac{\sqrt{-1}}{2\pi}\right)^n\int_M (\Delta \theta_X) F \; \omega^n \nonumber\\
 &= \left( \frac{\sqrt{-1}}{2\pi}\right)^n\int_M  \theta_X \Delta F \; \omega^n \nonumber\\
 &= \left( \frac{\sqrt{-1}}{2\pi}\right)^n\int_M  \theta_X S \; \omega^n - \bar S \int_M  \theta_X \; \omega^n \nonumber\\
 &= \left( \frac{\sqrt{-1}}{2\pi}\right)^n\int_M  n\theta_X \ric \wedge \omega^{n-1} - \frac{\bar S}{n+1} \left( \frac{\sqrt{-1}}{2\pi}\right)^n \int_M  (\theta_X + \omega)^{n+1}
\end{align}

\noindent which expresses the Calabi-Futaki invariant without explicit reference to $F$.  This expression (\ref{dividebyn}) also shows we may assume $\alpha = 0$ in (\ref{hodge}) for our purposes.

Let $A = (-\Delta \theta_X + \ric)$ and $B = (\theta_X + \omega)$.  By using the identity

\begin{equation}\label{combinatorial}
\sum_{j=0}^l (-1)^j {l \choose j} (l-2j)^k = \begin{cases} 0 \qquad k < l \textrm{ or } k=l+1 \\ 2^l l! \qquad k = l \end{cases}
\end{equation}

\noindent one checks that

\begin{align*}
\sum_{j=0}^n (-1)^j &{n \choose j} \int_M (A+(n-2j)B)^{n+1} - (-A + (n-2j)B)^{n+1} \\
&= \sum_{j=0}^n (-1)^j {n \choose j} \int_M \sum_{k=0}^{n+1} {n+1 \choose k} \left[ A^{n+1-k}(n-2j)^kB^k  - (-A)^{n+1-k}(n-2j)^kB^k\right] \\
&= 2^nn! \int_M (n+1)2AB^n \hspace{1.5 in} \textrm{(only $k=n$ is non-zero)} \\
&= 2^{n+1}(n+1)! \int_M  \left[ n\theta_X \, \ric \wedge \omega^{n-1} -\Delta \theta_X \, \omega^n \right] \\
&= 2^{n+1}(n+1)! \int_M n\theta_X \, \ric \wedge \omega^{n-1}\\
\end{align*}

%

\noindent Dividing this expression by $n!$ and substituting into the previous calculation (\ref{dividebyn}) yields

\begin{align}\label{expandedfutaki}
2^{n+1}(n+1)\fut (X,\Omega) = 
&\sum_{j=0}^n \frac{(-1)^j }{j!(n-j)!} \left( \frac{\sqrt{-1}}{2\pi}\right)^n\int_M [(-\Delta \theta_X + \ric +(n-2j)(\theta_X + \omega) )^{n+1} \nonumber \\
& \hspace{1.8 in} - (\Delta \theta_X - \ric + (n-2j)(\theta_X + \omega) )^{n+1}] \nonumber \\
&- \bar S \sum_{j=0}^{n+1} (-1)^j \frac{(n+1-2j)^{n+1}}{j!(n+1-j)!}  \left( \frac{\sqrt{-1}}{2\pi}\right)^n\int_M(\theta_X + \omega)^{n+1}
\end{align}

\noindent The point of expressing $\fut(X,\Omega)$ in this manner is that, as we will see, each integral (for each fixed $j$) may be realized as a Futaki-Morita type invariant, to which holomorphic localization applies.

\subsection{Holomorphic Localization}

We turn for a moment to a general description of Bott's holomorphic localization and its application to Futaki-Morita integrals.  Let $(M,g)$ be a Hermitian manifold and $E$ a holomorphic vector bundle on M with Hermitian metric $h$ and curvature $R(h) \in \Omega^{1,1}(\text{End}(E))$ of its Chern connection.  Suppose that there exists smooth $\theta(h) \in \Gamma (\text{End}(E))$ satisfying

\[
i_X R(h) = -\bar \partial \theta_X(h)
\]

Given an invariant polynomial 
\[
\phi: \mathfrak{gl}(\text{rank}(E),\C) \rightarrow \C
\]

\noindent of degree $n+k$, the \emph{Futaki-Morita} integral is defined as

\[
f_{\phi}(X) := \int_M \phi \left(\theta_X(h) + \frac{\sqrt{-1}}{2\pi}R(h)\right),
\]

\noindent which turns out to be independent of the chosen metrics \cite{futakimorita:1985}.

Futaki-Morita integrals may generally be computed via holomorphic localization: Define a $(1,0)$ form $\eta$ on $M/\textrm{Zero}(X)$ by 

\[
\eta(\cdot) = \frac{g(\cdot, \bar X)}{\|X\|^2}
\]

Bott's transgression argument \cite{bott:1967mich} \cite{griffithsharris} shows

\begin{align}\label{transgression}
f_{\phi}(X) 
= -\sum_{\lambda} \lim_{r \rightarrow 0^+} \int_{\partial B_{r}(Z_{\lambda})} \phi(\theta_X(h) + R(h)) \wedge \sum_{k = 0}^{n-1} (-1)^k \eta \wedge (\bar\partial \eta)^k
\end{align}

\noindent where $\{Z_{\lambda}\}$ is the zero locus of $X$ and $B_{r}(Z_{\lambda})$ is any small neighborhood of $Z_{\lambda}$.

We say that $\text{Zero}(X)$ is \emph{nondegenerate} when the endomorphism $L_{\lambda}(X)$ of the normal bundle $N_{\lambda}$ to $Z_{\lambda}$ induced by $X$ is invertible\footnote{One may verify this endomorphism is given by $L_{\lambda}(X)(Y) = (\nabla_YX)^{\perp}$}.  The work of Bott \cite{bott:1967jdg} essentially showed that when $\text{Zero}(X)$ is nondegenerate,

\begin{equation}\label{bott}
f_{\phi}(X) = \sum_{\lambda} \int_{Z_{\lambda}} \frac{\phi(\theta_X(h) + \frac{\sqrt{-1}}{2\pi}R(h))}{\det \left(L_{\lambda}(X) + \frac{\sqrt{-1}}{2\pi}K_{\lambda}\right)}
\end{equation}

\noindent where $K_{\lambda}$ is the curvature form of the connection induced on $N_{\lambda}$.

On the other hand, when $\text{Zero}(X)$ consists of isolated degenerate points:

\begin{theorem}[Cherveny \cite{cherveny:2016}]\label{degeneratelocalize}
If the zero locus of $X \in \mathfrak{h}$ is an isolated degenerate zero $p$ such that in local coordinates centered at $p$

\[
z_i^{\alpha_i + 1} = \sum b_{ij} X_j
\]

\noindent for some matrix $B = (b_{ij})$ of holomorphic functions, then

\begin{equation}\label{degeneratelocalization}
f_{\phi}(X) = \frac{1}{\prod \alpha_i !} \cdot \frac{\partial^{|\alpha|} \left( \phi(\theta_X(h)) \det B \right)}{\partial z_1^{\alpha_1} \cdots \partial z_n^{\alpha_n}} \bigg|_{z = 0}
\end{equation}

\noindent If $\text{Zero}(X)$ consists of multiple isolated degenerate zeros, then $f_{\phi}(X)$ is the sum of such contributions.
\end{theorem}

A special case of both (\ref{bott}) and (\ref{degeneratelocalization}) is when $p$ is an isolated nondegenerate zero: As $DX$ is invertible near $p$, take $B = (DX)^{-1}$ and $\alpha_i = 0$, giving

\begin{equation}\label{isolatednondeg}
f_{\phi}(X) = \phi(DX_p) \det B = \frac{\phi(DX_p)}{\det DX_p}
\end{equation}

Localization involving a positive dimensional degenerate zero locus is quite complicated and not understood in the general \kahler setting.  The calculations in this paper may be viewed as a step in this direction.

\subsection{Localization of $\fut(X, \Omega)$}

Returning to the localization of $\fut(\Omega, X)$, suppose without loss of generality that $\Omega = c_1(L)$ where $L$ is a positive line bundle.  Applying the above Futaki-Morita framework to the bundle $E = K_M^{\pm}\otimes L^{n-2j}$, standard computations yield 

\begin{gather*}
R_E = \pm \ric + (n-2j) \omega\\
i_X R_E = \pm i_X \ric + (n-2j)i_X \omega \\
-\bar \partial \theta_E = \pm i_X \ric - (n-2j)\bar \partial \theta_X \\
\theta_E = \mp \Delta \theta_X + (n-2j) \theta_X
\end{gather*}

Take $\phi$ to be the invariant polynomial $\phi(A) = \text{Tr}(A^{n+1})$.  The first integral in (\ref{expandedfutaki}) is then recognized to be 

\[
f_{\phi, E}(X) = \left( \frac{\sqrt{-1}}{2\pi}\right)^n\int_M (-\Delta \theta_X + \ric +(n-2j)(\theta_X + \omega) )^{n+1}  
\]

\noindent for $E = K_M\otimes L^{n-2j}$; the second is 

\[
f_{\phi, E}(X) = \left( \frac{\sqrt{-1}}{2\pi}\right)^n\int_M (\Delta \theta_X - \ric +(n-2j)(\theta_X + \omega) )^{n+1} 
\]

\noindent for $E = K_M^-\otimes L^{n-2j}$; and the third 

\[
f_{\phi, E}(X) = \left( \frac{\sqrt{-1}}{2\pi}\right)^n (n+1-2j)^{n+1} \int_M \left(\theta_X + \omega \right)^{n+1} 
\]

\noindent for $E = L^{n+1-2j}$.  The Calabi-Futaki invariant is thus fully expressible in terms of Futaki-Morita integral invariants.  

Applying localization (\ref{transgression}) to each Futaki-Morita invariant in this expression and using the combinatorial identity (\ref{combinatorial}) again to resolve summations yields

\[
\hspace{-.3 in}\fut(\Omega, X) = \lim_{r \rightarrow 0} \left( \frac{\sqrt{-1}}{2\pi}\right)^n \sum_{\lambda} \int_{\partial B_{r}(Z_{\lambda})} \hspace{-.15 in} \left[ (-\Delta \theta_X + \ric)(\theta_X + \omega)^n + \frac{\bar S (\theta_X + \omega)^{n+1}}{n+1} \right] \wedge \sum_{k=0}^{n-1} (-1)^k \eta \wedge (\bar \partial \eta)^k
\]

When $\text{Zero}(X)$ is nondegenerate, this expression was evaluated by Tian using (\ref{bott}), with explicit cohomological simplifications in the specific cases of nondegenerate isolated zeros, or nondegenerate zero loci on a K\"ahler surface, or the Fano case $\Omega = c_1(M)$.  See Theorem 6.3 in \cite{tian:1996}; also p. 31 in \cite{tian:2000}.

\subsection{Blowup Situation}

Our interest will be the blowup scenario where degenerate contributions to localization naturally arise.  To align notation with Li-Shi \cite{lishi:2015}, let $\mu = \frac{\bar S}{n}$ and define

\begin{gather*}
I_{Z_{\lambda}} := \sum_{\lambda} \lim_{r \rightarrow 0^+} \left( \frac{\sqrt{-1}}{2\pi}\right)^n \int_{\partial B_{r}(Z_{\lambda})} (\Delta \theta_X - \ric)(\theta_X + \omega)^n \wedge \sum_{k = 0}^{n-1} (-1)^k \eta \wedge (\bar\partial \eta)^k \\
J_{Z_{\lambda}} := - \sum_{\lambda} \lim_{r \rightarrow 0^+} \left( \frac{\sqrt{-1}}{2\pi}\right)^n \int_{\partial B_{r}(Z_{\lambda})} (\theta_X + \omega)^{n+1} \wedge \sum_{k = 0}^{n-1} (-1)^k \eta \wedge (\bar\partial \eta)^k\\
\fut_{Z_{\lambda}}(X, \Omega) := I_{Z_{\lambda}} - \frac{n\mu}{n+1} J_{Z_{\lambda}}
\end{gather*}

\noindent so that

\begin{equation}\label{futakidecomposition}
\fut(\Omega, X) = \sum_{\lambda} \fut_{Z_{\lambda}}(X, \Omega) = \sum_{\lambda} \left(I_{Z_{\lambda}} - \frac{n\mu}{n+1} J_{Z_{\lambda}} \right).
\end{equation}

\noindent Also define the summed contribution

\begin{gather*}
J_M(\Omega, X) := \sum_{\lambda} J_{Z_{\lambda}}(\Omega, X) = \left( \frac{\sqrt{-1}}{2\pi}\right)^n \int_M (\theta_X + \omega)^{n+1}
\end{gather*}

We close by expressing a relationship between local contributions to $\fut(X, \Omega)$ on $M$ and those for $\fut(\tilde X, \tilde \Omega)$ on $\tilde M$, where notation is as in the introduction. 


\begin{lemma}[Li-Shi \cite{lishi:2015}, Lemma 3.1]\label{contributionslemma} Let $X\in \mathfrak{h}$ vanish at $p$, $\tilde M$ be the blowup of $M$ at $p$, $\tilde \Omega$ be the \kahler class $\pi^*\Omega - \epsilon c_1(E)$, and $\tilde X$ be the extension of $X$ to $\tilde M$.  Define $\delta = \tilde \mu - \mu$.  Then
\[
\fut_{\tilde M}(\tilde \Omega, \tilde X) = \fut_M(\Omega, X) - \frac{n \delta}{n+1} J_M(\Omega, X) + \sum_{E} \fut_{\tilde Z_{\lambda}}(\tilde \Omega, \tilde X) + \frac{n\delta}{n+1} J_p(\Omega, X) - \fut_p(\Omega, X)
\]
\end{lemma}

The lemma is a consequence of the localization formula for the Futaki invariant (\ref{futakidecomposition}) and the above definitions after separating the fixed components of $\tilde X$ on $\tilde M$ into those contained in the exceptional divisor and those not.  The latter type is in one-to-one correspondence with the fixed components of $X$ on $M$ apart from $p$, and moreover these local Futaki invariants agree after an adjustment to $J_{Z_{\lambda}}$ by $\delta$.

Recall that $\theta_X$ is defined up to addition of a constant.  Without loss of generality, we may prove our main theorem under the simplifying assumption that this constant is chosen so $\theta_X$ has average value zero, and consequently $J_M(\Omega, X) = 0$.  As $p$ is nondegenerate, the term $\fut_p(\Omega, X)$ is immediate using (\ref{isolatednondeg}):

\begin{equation}\label{futakip}
\fut_p(\Omega, X) = I_p - \frac{n\mu}{n+1} J_p = \frac{\text{Tr}(A)}{\det A} \theta_p^n - \frac{n\mu}{n+1} \frac{\theta_p^{n+1}}{\det A}
\end{equation}

\noindent where we have simplified the notation using $A := DX_p$ and $\theta_p := \theta_X(p)$. 

With this choice of $\theta_X$, and in light of Lemma \ref{contributionslemma} and (\ref{futakip}),

\begin{lemma}\label{toprovelemma}
To prove Theorem \ref{maintheorem}, it is sufficient to show
\[
\frac{\text{Tr}(A) \theta_p^n}{\det A} - \frac{n(\mu + \delta)}{n+1} \frac{\theta_p^{n+1}}{\det A} - \sum_{j=1}^m \fut_{q_j}(\tilde \Omega, \tilde X) = n(n-1)\theta_p\epsilon^{n-1} + O(\epsilon^n)
\]
\noindent where $\{q_1, \dots, q_m\}$ are the isolated, possibly degenerate zeros of $\tilde X$ in $E$.
\end{lemma}






\section{Case I: Maximally Degenerate Zero}

In this section we will prove Theorem \ref{maintheorem} when $\tilde X$ has a single isolated zero in the exceptional locus, necessarily of maximal degeneracy.  The case of multiple zeros builds on these computations and will be given in the next section.

Let $p \in \text{Zero}(X)$ be the zero at which we will blow up $M$.  Choosing coordinates about $p$ such that $DX_p$ is in Jordan form, a maximal degenerate zero on the blowup corresponds to $DX_p$ having a single Jordan block.

To be precise, let $A$ denote the $n \times n$ Jordan matrix

\begin{equation}\label{jordanblock}
A = (A_{ij}) = \begin{bmatrix} a & 1 & 0 & \cdots & 0 \\ 0  & \ddots & \ddots & & \vdots \\ \vdots &  & \ddots & \ddots & 0 \\ \vdots &  & & \ddots & 1 \\ 0 & \cdots & \cdots & 0 & a \end{bmatrix}
\end{equation}

By Poincar\'e's Theorem \ref{poincare}, if $DX_p = A$ then $X$ is biholomorphically equivalent to its linearization on some neighborhood $U$ of $p$ provided that $a \neq 0$ (the zero is nondegenerate).  Therefore, without loss of generality, we may assume $X = \sum X_i \frac{\partial}{\partial z_i}$ is given on $U$ by

\begin{equation}\label{oneblock}
X = \sum_{i=1}^{n-1}(az_i + z_{i+1}) \frac{\partial}{\partial z_i} + az_n \frac{\partial}{\partial z_n}
\end{equation}

%
%
%
%
%
%
%
%

Following \cite{lishi:2015}, we now describe $X$'s natural extension $\tilde X$ to the blowup in local coordinates.  Let $\tilde U = \pi^{-1}(U)$ be the neighborhood of the exceptional divisor on $\tilde M = \textrm{Bl}_p M$ given by

\[
\tilde U = \{\left((z_1, \dots, z_n),[\eta_1, \dots, \eta_n] \right) | z_i\eta_j = z_j \eta_i\} \subseteq \C^n \times \C P^{n-1}
\]

Cover $\tilde U$ with charts $\tilde U_i = \{ \eta_i \neq 0\}$ having local coordinates 

\[
(u_1, \dots, u_n) = \left( \frac{\eta_1}{\eta_i}, \dots, z_i, \dots, \frac{\eta_n}{\eta_i} \right)
\]

\noindent Note that the slice of $\tilde U_i$ with $z_i = 0$ is just the standard cover for $E \cong \C P^{n-1}$.


In these coordinates the holomorphic extension $\tilde X$ of any $X \in \mathfrak{h}$ vanishing at $p$ to the blow-up is given by

\begin{equation}\label{xtildelocalcoords}
\tilde X \big|_{\tilde U_i} = X_i \frac{\partial}{\partial u_i} + \sum_{j \neq i} \frac{1}{u_i}(X_j - u_jX_i) \frac{\partial}{\partial u_j} 
\end{equation}

In particular, $\tilde X$ over $\tilde U_1$ for the one block case (\ref{oneblock}) presently being considered is

\begin{equation}\label{oneblockblowup}
\tilde X|_{\tilde U_1} = u_1(a + u_2) \frac{\partial}{\partial u_1} + \sum_{i = 2}^{n-1} \left[ (u_{i+1} - u_i u_2) \frac{\partial}{\partial u_i}\right] + (-u_2 u_n)\frac{\partial}{\partial u_n}.
\end{equation}

\noindent One may verify that the isolated zero at the origin in this chart, which we denote $q$, is indeed the only zero of $\tilde X$ in the exceptional divisor.  It is ``maximally degenerate" in the sense that zero is an eigenvalue of $D\tilde X_q$ with algebraic multiplicity $n-1$.  This is all geometrically obvious when one considers the action induced by $X$ on lines through $p$, where the zero locus of $\tilde X$ in the exceptional divisor corresponds to eigenspaces for $DX_p$.

We now choose a convenient metric in the class $\tilde \Omega = \pi^*\Omega - \epsilon c_1([E])$ following \cite{griffithsharris} (p. 185).  Denote by $B_r \subseteq U$ the ball of radius $r$ centered at $p$.  Suppose for simplicity $B_1 \subseteq U$, and let $\tilde B_1 = \pi^{-1}(B_1)$.  The fiber of $[E]$ over $\tilde B_1$ at a point is simply 

\[
[E]_{(z,\eta)} = \{\lambda(\eta_1, \dots, \eta_n), \lambda \in \C\}.
\]

Denote by $h_1$ the metric on $[E]|_{\tilde B_1}$ given by $\|(\eta_1, \dots, \eta_n)\|^2$.  Also let $\sigma \in H^0(\tilde M, \mathcal{O}([E]))$ be the holomorphic section whose zero divisor is $E$ and denote by $h_2$ the metric on $[E]|_{\tilde M/E}$ such that $\|\sigma\|_{h_2} \equiv 1$.  Finally choose a partition of unity $\{\rho_1, \rho_2\}$ subordinate to the cover $\{\tilde B_1, \tilde M / \tilde B_{1/2}\}$.  The metric 

\[
h = \rho_1 h_1 + \rho_2 h_2
\]

\noindent has nonzero curvature only on $\tilde B_1$.  Our \kahler form will be
\[
\tilde \omega = \pi^*\omega + \epsilon \partial \bar \partial \log h.
\]

A short calculation shows the holomorphy potential $\theta_{\tilde X}$ for $\tilde X$ relative to this $\tilde \omega$ on $\tilde U_1$ is

%
%

\begin{equation}\label{thetablowup}
\theta_{\tilde X} = \pi^*\theta_X - \epsilon \left[ a + \frac{u_2 + \bar u_2 u_3 + \bar u_3 u_4  + \dots + \bar u_{n-1} u_n}{1 + |u_2|^2 + \dots + |u_n|^2} \right]
\end{equation}

For brevity, denote $\theta_X (p)$ by $\theta_p$ and likewise for $\tilde \theta_q := \theta_{\tilde X} (q)$, so that 
\begin{equation}\label{thetaq}
\tilde \theta_q = \theta_p - a\epsilon.
\end{equation}

\noindent All derivatives of all orders of $\theta_{\tilde X}$ vanish at $q$ with the exception of
\begin{equation}\label{thetaqprime}
\frac{\partial \theta_{\tilde X}}{\partial u_2}(q) = -\epsilon.
\end{equation}

\noindent Another short calculation shows that

\begin{equation}\label{thetalaplacian}
\Delta_{\tilde \omega} \theta_{\tilde X} = a- (n-1)u_2
\end{equation}

By Theorem \ref{degeneratelocalize}, the local contribution to any Futaki-Morita integral in the present situation (and in particular the Futaki invariant) is given by the residue formula

\begin{equation}\label{localizeoneblock}
\text{Res}_q \phi = \frac{1}{\prod (\alpha_j - 1)!} \frac{\partial^{|\alpha|}(\phi \det B)}{\partial u_1^{\alpha_1 - 1} \partial u_2^{\alpha_2 - 1} \cdots \partial u_n^{\alpha_n - 1}} \bigg|_q
\end{equation}

\noindent where $\alpha_j$ are natural numbers and $B = (b_{ij})$ is an $n \times n$ matrix such that 
\begin{equation}\label{hilbert}
u_j^{\alpha_j} = \sum b_{ij} \tilde X_i.
\end{equation}

We now construct $B$ in order to calculate (\ref{localizeoneblock}).  Choose $k$ such that $2^k < n \leq 2^{k+1}$.  It is straightforward to verify for (\ref{oneblockblowup}) that 

\begin{align*}
u_1 &= \left[ \frac{1}{a} - \sum_{i=2}^n \left(\frac{-1}{a}\right)^i u_i \right] \tilde X_1 + \sum_{i = 2}^n\left[- u_1  \left(\frac{-1}{a}\right)^i \right] \tilde X_i \\
u_2^n &= \sum_{i = 2}^n \left[- u_2^{n-i} \right] \tilde X_i \\
\end{align*}

For $j = 3, \dots, n$, one calculates

\[
u_j^{2^{k+1}} = \sum_{l = 2}^{j-1} \left[ -u_2^{(j-1)2^{k+1} - l} + u_2^{(j-l-1)2^{k+1}} \prod_{i=0}^k(u_{l+1}^{2^i} +u_2^{2^i}u_{l}^{2^i}) \right] \tilde X_l + \sum_{l = j}^n -u_2^{(j-1)2^{k+1} - l} \tilde X_l\]

\noindent The idea is to repeatedly factor $u_j^{2^{k+1}} - (u_2u_{j-1})^{2^{k+1}}$ into binomials, one of which is eventually $\tilde X_j$, and insert the above expression for $u_2^n$ in terms of the $\tilde X_i$.  These relations contain the information necessary to construct $B$ for (\ref{hilbert}) with parameters $\alpha_1 = 1, \alpha_2 = n, \alpha_j = 2^{k+1}$ for $j = 3, \dots, n$ (these are certainly not minimal $\alpha_i$ for every $n$, but are convenient for a general setup).

With these choices, the determinant of $B$ is found by row reduction to be

\begin{equation}\label{detB}
\det B = (-1)^{n-1}\left[ \frac{1}{a} - \sum_{i=2}^n \left(\frac{-1}{a}\right)^i u_i \right] \prod_{j=3}^{n} \prod_{i=0}^{k}(u_j^{2^i} +u_2^{2^i}u_{j-1}^{2^i})
\end{equation}

Applying (\ref{localizeoneblock}), the residue of interest is

\[
\textrm{Res}_q \phi = \frac{1}{(n-1)![(2^k)!]^{n-2}} \cdot \frac{\partial^{|a|}(\phi \det B)}{(\partial u_2)^{n-1} (\partial u_3)^{2^k} \cdots (\partial u_n)^{2^k}} \bigg|_q
\]

Since $\phi$ is a function of $\theta_{\tilde X}$, which depends only on $u_2$ and its derivatives in our case, all other derivatives must be applied to $\det B$.  Doing so, the coefficient of $u_3^{2^k} \dots u_n^{2^k}$ in $\det B$ is found to be

\begin{equation}\label{detB1}
(-1)^{n-1} \sum_{i=0}^{n-1} \frac{(-u_2)^i}{a^{i+1}}
\end{equation}

\noindent so that the residue after taking appropriate derivatives and evaluating at $u_3 = ... = u_n = 0$ is

\begin{align*}
\textrm{Res}_q \phi &= \frac{1}{(n-1)!} \frac{\partial}{(\partial u_2)^{n-1}} \left(\phi \sum_{i=0}^{n-1} (-1)^{n-1} \frac{(-u_2)^i}{a^{i+1}}\right) \bigg|_{q} \\
&= \frac{1}{(n-1)!} \sum_{j=0}^{n-1} \binom{n-1}{j} \cdot \frac{(-1)^{n-1}(-1)^{n-1-j}(n-1-j)!}{a^{n-1-j+1}} \cdot \frac{\partial^j \phi}{(\partial u_2)^j}(0)  \\
&= \sum_{j=0}^{n-1} \frac{(-1)^j}{j!a^{n-j}} \frac{\partial^j \phi}{(\partial u_2)^j}(0)
\end{align*}

We have shown

\begin{lemma}\label{maxdegenlemma}
If $\phi$ is an invariant polynomial whose value depends only on $u_2$ in the above situation, and $DX_p$ is a single Jordan block with eigenvalue $a$, then the residue contribution to the Futaki-Morita of the blowup at $p$ at the unique isolated zero $q$ is

\[
\textrm{Res}_q \phi = \sum_{i=0}^{n-1} \frac{(-1)^i}{i!a^{n-i}} \frac{\partial^i \phi}{(\partial u_2)^i}(0)
\]
\end{lemma}

\noindent When $n = 2$ we recover Lemma 3.6 of \cite{lishi:2015}, which is the main calculation of the paper and obtained by brute force calculus.

We can now give a direct proof of Theorem \ref{maintheorem} in the case $DX_p = A$ by verifying the identity in Lemma \ref{toprovelemma}.  The term to calculate is $\fut_q(\tilde \Omega, \tilde X) = I_q - \frac{n}{n+1}(\delta + \mu)J_q$.

%

For $J_q$, apply Lemma \ref{maxdegenlemma} with $\phi = \theta_{\tilde X}^{n+1}$.  By (\ref{thetaq}) and (\ref{thetaqprime}),

\[
J_q = \sum_{i=0}^{n-1} \frac{(n+1)!(\theta_p - a\epsilon)^{n+1-i}\epsilon^i}{i!a^{n-i}(n+1-i)!}
\]

\noindent Binomial expansion and interchanging summations yields

\[
J_q = (n+1)! \sum_{j=2}^{n+1} \sum_{i=0}^{n+1-j} \frac{(-1)^{n+1-i-j}}{i!j!(n+1-i-j)!a^{j-1}}\theta^j_p \epsilon^{n+1-j}  + O(\epsilon^n).
\]

The coefficient of $\epsilon^{n+1-j}$ is proportional to $\sum_{i=0}^{n+1-j} (-1)^i {n+1-j \choose i}$, which vanishes by symmetry of binomial coefficients unless $j = n+1$.  It follows that

\[
J_q = \frac{\theta_p^{n+1}}{a^n} + O(\epsilon^n).
\]


On the other hand, to calculate $I_q$ use $\phi = (-\Delta \theta_{\tilde X})\theta_{\tilde X}^n$ in Lemma \ref{maxdegenlemma}.  Again by (\ref{thetaq}), (\ref{thetaqprime}), and (\ref{thetalaplacian}),

\[
I_q = \sum_{i=0}^{n-1} \frac{n!\epsilon^{i-1}}{i!a^{n-i}(n-i)!} \left[ \frac{(n-1)i(\theta_p - a\epsilon)^{n+1-i}}{n-i+1} + a(\theta_p - a\epsilon)^{n-i}\epsilon \right]
\]

\noindent By a similar expansion, the summed second term simplifies to 

\[
\frac{\theta_p^n}{a^{n-1}} + O(\epsilon^n) 
\]

\noindent while the first simplifies to

\[
\frac{(n-1)\theta_p^n}{a^{n-1}} - n(n-1)\theta_p \epsilon^{n-1} + O(\epsilon^n) 
\]

\noindent yielding

\[
I_q = \frac{n\theta_p^n}{a^{n-1}} - n(n-1)\theta_p \epsilon^{n-1} + O(\epsilon^{n}) 
\]

Putting everything together in Lemma \ref{toprovelemma},

\begin{align*}
\fut_p(X, \Omega) &- \fut_q(\tilde X, \tilde \Omega) - \frac{n \delta}{n+1} J_p \\
&= \left[ \frac{an \theta_p^n}{a^n} - \frac{n(\mu + \delta)}{n+1} \frac{\theta_p^{n+1}}{a^n} \right] - \left[\frac{n\theta_p^n}{a^{n-1}} - n(n-1)\theta_p \epsilon^{n-1} - \frac{n(\mu + \delta)}{n+1}\frac{\theta^{n+1}_p}{a^n} + O(\epsilon^{n}) \right] \\
&= n(n-1)\theta_p \epsilon^{n-1} + O(\epsilon^n) 
\end{align*}


\noindent which completes our verification of Theorem \ref{maintheorem} in this special case.


\section{Case II: Multiple Degenerate Zeros}

In this section we complete the proof of Theorem \ref{maintheorem}.  Suppose the linearization of $X$ at $p$ now has multiple Jordan blocks.  Hypothesis ($\star$) means that each Jordan block corresponds to an isolated degenerate zero in the exceptional divisor, and by a change of coordinate we may assume a particular degenerate zero corresponds to the first block.  We extend the computations from Section 3 to calculate the contribution to the Futaki invariant from the first block under the influence of other blocks.  The net contribution from all degenerate zeros is then a sum given by symmetrizing that formula, which we evaluate using Lemma \ref{lemmagk} (proved via integration of meromorphic differentials in the appendix).

Suppose coordinates centered at $p \in \text{Zero}(X)$ have been chosen such that the linearization $DX$ of $X$ is in Jordan form at $p$:

\[
DX_p = \begin{bmatrix} A_1 & 0 & \cdots & 0 \\ 0 & A_2 & \ddots & \vdots \\ \vdots & \ddots & \ddots & 0 \\ 0 & \cdots & 0 & A_m \end{bmatrix}
\]

\noindent Each Jordan block $A_i$ is of the form (\ref{jordanblock}) with diagonal entries $a_i$ and dimension $n_i$.  By Lemma \ref{normalformthrm}, we may assume $X$ is biholomorphic to its linearization near $p$.

Let $s_j = \sum_{k = 1}^j n_k$, so that $s_m = n$.  In the coordinates introduced in Section 3,

\begin{equation}\label{generallift}
\tilde X \big|_{\tilde U_1} = u_1(a_1 + u_2) \frac{\partial}{\partial u_1} + \sum_{j = 1}^m \left[ \sum_{i = s_{j-1} + 1}^{s_j - 1} \left[u_{i + 1} - u_i(u_2 + a_1 - a_j) \right] \frac{\partial}{\partial u_i} - u_{s_j}(u_2 + a_1 - a_j) \frac{\partial}{\partial u_{s_j}} \right]
\end{equation}

\noindent (which of course reduces to the one block formula (\ref{oneblockblowup}) when $m = 1$).  The zero at the origin in $\tilde U_1$ will be denoted $q_1$.

Our first task is to construct the appropriate $B$ to apply Theorem \ref{degeneratelocalize}.  Let $k$ be the natural such that $2^k < n_1 \leq 2^{k + 1}$.  The work of Section 3 constructs a matrix $B_1$ expressing powers of $u_1, \dots, u_{s_1}$ in terms of $\tilde X_1, \dots, \tilde X_{s_1}$.  For $1 < j \leq m$, one checks

\begin{equation}\label{xtildesj}
u_{s_j} = \tilde X_{s_j} \frac{\prod_{i=0}^k [(a_j - a_1)^{2^i} + u_2^{2^i}]}{(a_j - a_1)^{2^{k+1}}} + u_2^{2^{k+1}} u_{s_j} 
\end{equation}

The idea here is to factor $u_{s_j}(u_2^{2^{k+1}} - (a_{s_j} - a_1)^{2^{k+1}})$.  We are then done since $u_2^{2^{k+1}}$ is a known linear combination of $\tilde X_1, \dots, \tilde X_{s_1}$ from Section 3.  For $u_{s_{j-1}} < u_i < u_{s_j}$,

\begin{equation}\label{uigeneral}
u_i = \frac{\prod_{i=0}^k [(a_j - a_1)^{2^i} + u_2^{2^i}]}{(a_j - a_1)^{2^{k+1}}} \tilde X_i + u_2^{2^{k+1}}u_i - \frac{\prod_{i=0}^k [(a_j - a_1)^{2^i} + u_2^{2^i}]}{(a_j - a_1)^{2^{k+1}}} u_{i+1}
\end{equation}

\noindent so that, in light of (\ref{xtildesj}), we may recursively solve to obtain 

\begin{equation}\label{uigeneralB}
u_i = \frac{\prod_{i=0}^k [(a_j - a_1)^{2^i} + u_2^{2^i}]}{(a_j - a_1)^{2^{k+1}}} \tilde X_i + \text{linear combination of $\{\tilde X_1, \dots, \tilde X_{n_1}, \tilde X_{i+1}, \dots, X_{s_j}\}$}.
\end{equation}


It follows that $B$ has the form

\begin{equation}\label{generalB}
B = \begin{bmatrix} B_1 & 0 & 0 & 0 \\ * & B_2 & 0 & 0 \\ * & 0 & \ddots & 0 \\ * & 0 & 0 & B_m \end{bmatrix}
\end{equation}

\noindent where $B_1$ was constructed in the previous section, and each $B_j$ for $j > 1$ is upper triangular with the entries recoverable from (\ref{xtildesj}) and (\ref{uigeneral}).

We have calculated $\det B_1$ in (\ref{detB}), while for $j > 1$,

\[
\det B_j = \left(\frac{\prod_{i=0}^k [(a_j - a_1)^{2^i} + u_2^{2^i}]}{(a_j - a_1)^{2^{k+1}}}\right)^{n_j}
\]

\noindent Clearly the only derivative that may be applied to $\det B_j$ for $j> 1$ is the $u_2$-derivative.  The $i$-th $u_2$-derivative of $\det B_j$ evaluated at $u_2 = 0$ is

\[
\frac{\partial}{(\partial u_2)^i} \det B_j = \frac{(n_j + i -1)!}{(n_j - 1)!(a_j - a_1)^{n_j +i}}
\]

Putting everything into (\ref{localizeoneblock}), the residue of interest

\begin{align}\label{generalresidue}
\textrm{Res}_{q_1} \phi &= \frac{1}{(n_1-1)!} \frac{\partial}{(\partial u_2)^{n_1-1} }\left(\phi \prod_{j=1}^m \det B_j \right) \nonumber\\
&= \frac{1}{(n_1 - 1)!} \sum \binom{n_1 - 1}{i, \mu_1, \dots, \mu_m} \cdot \prod_{j=1}^m \frac{\partial \det B_j}{(\partial u_2)^{\mu_j}} \cdot \frac{\partial \phi}{(\partial u_2)^i}  \bigg|_{u = 0}  \nonumber \\
&= \sum_{i = 0}^{n_1 - 1} \sum_{\mu } \frac{1}{i!} \left( \prod_{j=1}^m \frac{1}{\mu_j!} \frac{\partial \det B_j}{(\partial u_2)^{\mu_j}} \right) \cdot \frac{\partial \phi}{(\partial u_2)^i}  \bigg|_{u = 0} \nonumber \\
&= \sum_{i = 0}^{n_1 - 1} \sum_{\mu } \frac{(-1)^{n_1 +\mu_1 -1}}{i! a_1^{\mu_1 + 1}}  \left( \prod_{j=2}^m \frac{\binom{n_j + \mu_j - 1}{\mu_j}}{(a_j - a_1)^{n_j + \mu_j}} \right) \cdot \frac{\partial \phi}{(\partial u_2)^i}  \bigg|_{u = 0}
\end{align}

\noindent where $\mu = (\mu_1, \dots, \mu_m)$ runs over all partitions of $n_1 - i - 1$ of length $m$.  In the last line we have used (\ref{detB1}).  This generalizes Lemma \ref{maxdegenlemma} ($m = 1$ is the lemma).

We now complete the proof of Theorem \ref{maintheorem} by verifying the identity in Lemma \ref{toprovelemma}.  The holomorphy potential of $\tilde X$ as in (\ref{generallift}) is 

\[
\theta_{\tilde X} = \pi^*\theta_X - \epsilon \left[a_1 + \frac{u_2 + \bar u_2 u_3 + \dots + \bar u_{n-1} u_n + \sum_{j=1}^m \sum_{i=s_{j-1} + 1}^{s_j} (a_j - a_1)|u_i|^2}{1 + |u_2|^2 + \dots + |u_n|^2} \right],
\]

\noindent generalizing (\ref{thetablowup}).  As in (\ref{thetaq}) and (\ref{thetaqprime}), $\theta_{\tilde X}$ satisfies

\begin{gather}\label{generalthetavalues}
\begin{split}
\tilde \theta_{q_1} = \theta_p - a_1\epsilon \\
\frac{\partial \theta_{\tilde X}}{\partial u_2}(q_1) = -\epsilon
\end{split}
\end{gather}

\noindent while all other derivative of $\theta_{\tilde X}$ vanish at $q_1$, and the Laplacian generalizing (\ref{thetalaplacian}) is

\begin{equation}\label{generallaplacian}
\Delta_{\tilde \omega} \theta_{\tilde X} = \text{Tr}(A) - (n-1)(u_2 + a_1).
\end{equation}

\noindent With (\ref{generalthetavalues}) in mind, $J_{q_1}$ is calculated by applying (\ref{generalresidue}) to $\phi =  \theta^{n+1}_{\tilde X}$:

\begin{align*}
J_{q_1} &= \sum_{i = 0}^{n_1 - 1} \sum_{\mu } \frac{(-1)^{n_1 +\mu_1 -1}}{i! a_1^{\mu_1 + 1}}  \left( \prod_{j=2}^m \frac{\binom{n_j + \mu_j - 1}{\mu_j}}{(a_j - a_1)^{n_j + \mu_j}} \right) \cdot \frac{(n+1)! (-\epsilon)^i(\theta_p - a_1 \epsilon)^{n+1-i}}{(n+1-i)!}
 \end{align*}
 
\noindent where $\mu$ is still runs over partitions of $n_j - i - 1$ of length $m$.

Interchanging $1 \leftrightarrow j$ gives the sum over all zeros $\{q_1, \dots, q_m\}$ in the exceptional divisor to be

\begin{align}\label{sumofJresidues}
\sum_j J_{q_j} &= \sum_{j=1}^m \sum_{i = 0}^{n_j - 1} \frac{ (n+1)! (-\epsilon)^{i}(\theta_p - a_j \epsilon)^{n+1-i}}{i!(n+1-i)!}
		 			\sum_{\mu} \frac{(-1)^{n_j +\mu_j -1}}{a_j^{\mu_j + 1}} 
					\left( \prod_{l \neq j}^m \frac{\binom{n_l + \mu_l - 1}{\mu_l}}{(a_l - a_j)^{n_l + \mu_l}} \right) \nonumber \\
&= \sum_{j=1}^m \sum_{i=0}^{n_j-1} \sum_{k=0}^{n+1-i} \frac{(-1)^{n-k}(n+1)!}{k!i!(n+1-i-k)!} \theta_p^k \epsilon^{n+1-k}\sum_{\mu} \frac{(-1)^{n_j +\mu_j}}{a_j^{\mu_j+i+k-n}}  \left( \prod_{l \neq j}^m \frac{\binom{n_l + \mu_l - 1}{\mu_l}}{(a_l - a_j)^{n_l + \mu_l}}  \right) \nonumber \\
&= \sum_{k=0}^{n+1} (-1)^{n-k} \binom{n+1}{k} \theta_p^k \epsilon^{n+1-k} G_k
\end{align}

\noindent where

\[
G_k = \sum_{j=1}^m \sum_{i=0}^{n+1-k} \binom{n+1-k}{i} \sum_{\mu} \frac{(-1)^{n_j +\mu_j}}{a_j^{\mu_j+i+k-n}}  \left( \prod_{l \neq j}^m \frac{\binom{n_l+ \mu_l - 1}{\mu_l}}{(a_l - a_j)^{n_l + \mu_l}}  \right).
\]

\noindent We have simplified the notation by allowing $i$ to run into values that make the partition $\mu$ of $n_j - i - 1$ a partition of a negative number.  The term is understood to be zero when this happens.

\begin{lemma}\label{lemmagk}
\[
G_k =
\begin{cases}
\frac{-1}{\det A} & k = n+1\\
0 & 1 < k < n+1\\
(-1)^n & k = 1
\end{cases}
\]
\end{lemma}

\noindent See the appendix in Section \ref{appendix} for a proof via integration of meromorphic differentials.  It follows from Lemma \ref{lemmagk} that 

\[
\sum_j J_{q_j} = \frac{\theta_p^{n+1}}{\det A} + O(\epsilon^n)
\]

Likewise, to compute the contribution $I_{q_j}$ from each fixed point $q_j$, use $\phi = (-\Delta \theta_{\tilde X})\theta_{\tilde X}^n$ in (\ref{generalresidue}) along with (\ref{generalthetavalues}) and (\ref{generallaplacian}).  The contribution is

\begin{align}\label{sumofIresidues}
\sum_j I_{q_j} = \sum_{j=1}^m \sum_{i = 0}^{n_j - 1} & \sum_{\mu } \frac{(-1)^{n_1 +\mu_1 -1}}{i! a_j^{\mu_1 + 1}} \left( \prod_{l \neq j}^m \frac{\binom{n_j + \mu_j - 1}{\mu_j}}{(a_j - a_l)^{n_j + \mu_j}} \right) \\
& \cdot \left[ \frac{-n!(\text{Tr}(A) - (n-1)a_j) (-\epsilon)^i(\theta_p - a_j \epsilon)^{n-i}}{(n-i)!} + \frac{i(n-1)n! (-\epsilon)^{i-1}(\theta_p - a_j \epsilon)^{n-i+1}}{(n-i+1)!} \right] \nonumber
\end{align}

This summation is of the form

\[
\sum_{k=1}^n (-1)^{n-k+1} \binom{n}{k} \theta_p^k \epsilon^{n-k} [G_k' + G_k'' + G_k'''] + O(\epsilon^n)
\]

\noindent where

\begin{align*}
G_k' &= \text{Tr}(A) \sum_{j=1}^m \sum_{i=0}^{n-k} \binom{n-k}{i} \sum_{\mu} \frac{(-1)^{n_j +\mu_j}}{a_j^{\mu_j+i+k-n+1}}  \left( \prod_{l \neq j}^m \frac{\binom{n_l+ \mu_l - 1}{\mu_l}}{(a_l - a_j)^{n_l + \mu_l}}  \right) \\
G_k'' &= -(n-1) \sum_{j=1}^m \sum_{i=0}^{n-k} \binom{n-k}{i} \sum_{\mu} \frac{(-1)^{n_j +\mu_j}}{a_j^{\mu_j+i+k-n}}  \left( \prod_{l \neq j}^m \frac{\binom{n_l+ \mu_l - 1}{\mu_l}}{(a_l - a_j)^{n_l + \mu_l}}  \right)\\
G_k''' &= -(n-1) \sum_{j=1}^m \sum_{i=1}^{n-k+1} \binom{n-k}{i-1} \sum_{\mu} \frac{(-1)^{n_j +\mu_j}}{a_j^{\mu_j+i+k-n}}  \left( \prod_{l \neq j}^m \frac{\binom{n_l+ \mu_l - 1}{\mu_l}}{(a_l - a_j)^{n_l + \mu_l}}  \right)
\end{align*}

\noindent and, as with $G_k$ above, any index $i$ producing a partition of a negative number contributes zero.

Clearly $G'_k = \text{Tr}(A)G_{k+1}$, while for each $1 \leq k \leq n$, combinatorial manipulation shows

\[
G_k'' + G_k''' = -(n-1)G_k.
\]

\noindent By Lemma \ref{lemmagk}, (\ref{sumofIresidues}) evaluates to

\[
\sum_j I_{q_j} = \frac{\text{Tr}(A)}{\det A} \theta^n_p - n(n+1)\theta_p \epsilon^{n-1} + O(\epsilon^n).
\]

Putting these results into Lemma \ref{toprovelemma},

\begin{align*}
&\fut_p (\Omega, X) - \sum_{j=1}^m \left(I_{q_j} - \frac{n \tilde \mu}{n+1} J_{q_j} \right) - \frac{n\delta}{n+1} J_p(\Omega, X)\\
&= \left(\frac{\text{Tr}(A)\theta_p^n}{\det A} - \frac{n(\mu+ \delta)}{n+1} \frac{\theta_p^{n+1}}{\det A} \right) - \left(\frac{\text{Tr}(A)\theta^n_p}{\det A}  - n(n+1)\theta_p \epsilon^{n-1} - \frac{n \tilde \mu}{n+1} \frac{\theta_p^{n+1}}{\det A} + O(\epsilon^n) \right) \\
&= n(n-1)\theta_X(p)\epsilon^{n-1} + O(\epsilon^n)
\end{align*}

\noindent which concludes the verification of Theorem \ref{maintheorem}.

\section{Normal Forms at Singularities}

Sections 3 and 4 rely on the assumption that $X$ is holomorphically equivalent to its linearization on a neighborhood of $p$, which is in general not true (not even smoothly).  Our main result of this section is that for the purposes of proving Theorem \ref{maintheorem} it is sufficient to assume such a normal form.

We recall a well-known condition originally due to Poincar\'e under which this assumption is true.  A vector $\lambda = (\lambda_1, \dots, \lambda_n) \in \C^n$ is called \emph{resonant} if there exists a relation 

\[
\lambda_k = \sum m_i \lambda_i
\]

\noindent where $m_i$ are nonnegative integers and $\sum m_i \geq 2$.  The vector $\lambda$ is said to \emph{belong to the Poincar\'e domain} if the convex hull of $\lambda_1, \dots, \lambda_n$ in $\C$ does not contain the origin.

\begin{theorem}[Poincar\'e \cite{arnold:1988}]\label{poincare}
If the eigenvalues of the linear part of a holomorphic vector field at a zero belong to the Poincar\'e domain and are nonresonant, then the vector field is biholomorphically equivalent to its linearization on a neighborhood of that zero.
\end{theorem}

This condition applies to the one non-trivial Jordan block needed in Li-Shi \cite{lishi:2015} and is key to simplifying their calculation.  It remains valid more generally for the case of a single Jordan block in Section 3, but does not apply in Section 4 in general.

We now show that it is sufficient to prove Theorem \ref{maintheorem} under the assumption $X$ is locally biholomorphic to its linearization.  To be precise, let $Y$ be a holomorphic vector field on $M$ with an isolated nondegenerate zero at $p$ satisfying ($\star$), so that in coordinates centered at $p$,

\begin{equation}\label{localX}
Y = \sum_j (A_{ij}z_i + O(z_kz_l) )\frac{\partial}{\partial z_j}
\end{equation}

Let $X$ denote the linearization of $Y$, defined on a neighborhood of $p$:

\[
X = \sum (A_{ij}z_i)\frac{\partial}{\partial z_j}
\]

\noindent The extension $\tilde X$ of $X$ to $\tilde M = Bl_pM$ is defined on a neighborhood of the exceptional divisor (see (\ref{xtildelocalcoords}) for the explicit local formula).

\begin{lemma}\label{normalformthrm}
Let $X, Y, \tilde X$, and $\tilde Y$ be as above.  The zero loci of $\tilde X$ and $\tilde Y$ on a neighborhood of the exceptional divisor $E$ agree, and for each isolated zero $q \in E$,

\[
\fut_q(\tilde X, \tilde \Omega) = \fut_q(\tilde Y, \tilde \Omega)
\]

\end{lemma}

\begin{proof}

In coordinates centered at $p$, $Y$ has the form in (\ref{localX}).  As $p$ is an isolated and nondegenerate zero of $Y$, it is also isolated and nondegenerate for $X$.  From expression (\ref{xtildelocalcoords}) it is clear that any zeros of $\tilde Y$ in a neighborhood of E must be on E (one must have $Y_1 = \dots = Y_n = 0$, and this can only happen on $\pi^{-1}(p)$).  Likewise for the zeros of $\tilde X$.  But the zeros on the exceptional divisor correspond with the eigenspaces of $DX_p = DY_p$, and so the zero loci agree as claimed.  In particular, the zeros of $\tilde Y$ are isolated iff the zeros of $\tilde X$ are isolated.

The local Futaki invariant for $X$ was calculated in Section 4 using (\ref{generalresidue}).  We will show that adding in higher order terms to form $Y$ changes $\theta_{\tilde X}, \Delta \theta_{\tilde X}$, and $\det B$ each by $O(u_1)$, so that

\begin{align}\label{normalwanttoshow}
\begin{split}
-(\Delta \theta_{\tilde Y})\theta_{\tilde Y}^n \det C &= -(\Delta \theta_{\tilde X})\theta_{\tilde X}^n \det B + O(u_1)  \\
\theta_{\tilde Y}^{n+1} \det C &= \theta_{\tilde X}^{n+1} \det B + O(u_1)
\end{split}
\end{align}

\noindent and moreover still no $u_1$ derivatives will be involved.  The lemma then follows, as $I_q$ and $J_q$ are calculated via (\ref{degeneratelocalization}) with no $u_1$-derivatives applied to $-(\Delta \theta_{\tilde Y})\theta_{\tilde Y}^n \det C$ and $\theta_{\tilde Y}^{n+1} \det C$, respectively, followed by evaluation involving $u_1 = 0$.

By assumption,

\[
Y_i = X_i + O(z_kz_l)
\]

\noindent for each $i = 1, \dots, n$.  Using the coordinates in Section 3 in which the zero of $\tilde Y$ is at the origin in $\tilde U_1$, by (\ref{xtildelocalcoords}) 
\[
\tilde Y_1 = \tilde X_1 + O(u_1^2)
\]

\noindent while for $2 \leq i \leq n$,
\begin{align*}
\tilde Y_i &= \frac{1}{u_1}[A_{ij}z_j + O(u_1^2) - u_i(A_{ij}z_j + O(u_1^2))] \\
&= \tilde X_i + O(u_1)
\end{align*}

Straightforward calculation then gives
\begin{align}\label{normallaplacian}
\Delta_{\tilde \omega} \theta_{\tilde Y} &= \sum_i \frac{\partial \tilde Y_i}{\partial u_i} \nonumber \\
 &= \frac{\partial}{\partial u_1} [\tilde X_1 + O(u_1^2)] + \sum_{i > 1} \frac{\partial}{\partial u_i} [\tilde X_i + O(u_1)] \nonumber\\
 &= \Delta_{\tilde \omega} \theta_{\tilde X} + O(u_1)
\end{align}

Next we consider the holomorphy potential itself.  On $U$,

\[
\bar \partial \theta_X = \bar \partial \theta_Y + O(z_l z_k)
\]

\noindent so that
\begin{align}\label{normaltheta}
\bar \partial \theta_{\tilde X} &= \pi^* \bar \partial \theta_X - \bar \partial ( \tilde X \log h) \nonumber \\
\bar \partial \theta_{\tilde X} &= \pi^* \bar \partial (\theta_X + O(z_kz_l)) - \bar \partial ( \tilde Y \log h) + \bar \partial (O(u_1)) \nonumber\\
\theta_{\tilde X}  &= \theta_{\tilde Y}  + O(u_1)
\end{align}

Lastly, let $B = (b_{ij})$ be the matrix of holomorphic functions near $q$ constructed in Sections 3 and 4 satisfying $u_j^{\alpha_j} = \sum_j b_{ij} \tilde X_j$.  We will show that there is similarly a matrix $C = (c_{ij})$ such that $u_j^{\alpha_j} = \sum_j c_{ij} \tilde Y_j$ for the same values of $\alpha_j$ (and in particular $\alpha_1 = 1$), such that 

\begin{equation}\label{normaldet}
\det C = \det B.
\end{equation}

Theorem \ref{poincare} does adapt to the first Jordan block of $DY_p$, providing a holomorphic coordinate system in which $Y_1, \dots, Y_{n_1}$ agree with $X_1, \dots, X_{n_1}$, and consequently $\tilde Y_1, \dots, \tilde Y_{n_1}$ agree with $\tilde X_1, \dots, \tilde X_{n_1}$ by (\ref{xtildelocalcoords}).  As a result, the matrix $C$ sought begins with a block $C_1$ identical to block $B_1$ in (\ref{generalB}).

For the $u_i$ where $i> n_1$, it is known from (\ref{uigeneralB}) that

\[
u_i = \frac{\prod_{i=0}^k [(a_j - a_1)^{2^i} + u_2^{2^i}]}{(a_j - a_1)^{2^{k+1}}} \tilde X_i + \text{linear combination of }\{\tilde X_1, \dots, \tilde X_{n_1}, \tilde X_{i+1}, \dots, X_{s_j}\}.
\]

\noindent Using that $\tilde Y_j = \tilde X_j + O(u_1)$ and that $u_1 = \sum_{j=1}^{n_1} c_{1j} \tilde Y_j$,

\[
u_i = \frac{\prod_{i=0}^k [(a_j - a_1)^{2^i} + u_2^{2^i}]}{(a_j - a_1)^{2^{k+1}}}\tilde Y_i + \text{linear combination of }\{\tilde Y_1, \dots, \tilde Y_{n_1}, \tilde Y_{i+1}, \dots, Y_{s_j}\}.
\]

\noindent It follows that the diagonal entries for $C_j$ are the same as the diagonal entries for $B_j$ for $j>1$ and that $C$ has the same form of $B$ in (\ref{generalB}), which is sufficient to conclude $\det C = \det B$.

With (\ref{normallaplacian}), (\ref{normaltheta}), and (\ref{normaldet}), we have shown (\ref{normalwanttoshow}) and the proof is complete.

%
%
%
%
%
%
%
%
%
%
%

\end{proof}

\section{Appendix: Proof of Lemma \ref{lemmagk}}\label{appendix}

We now prove

\begin{lemma}[Lemma \ref{lemmagk}]  Let 
\[
G_k = \sum_{j=1}^m \sum_{i=0}^{n+1-k} \binom{n+1-k}{i} \sum_{\mu} \frac{(-1)^{n_j +\mu_j}}{a_j^{\mu_j+i+k-n}}  \left( \prod_{l \neq j}^m \frac{\binom{n_l+ \mu_l - 1}{\mu_l}}{(a_l - a_j)^{n_l + \mu_l}}  \right)
\]

\noindent Then

\[
G_k =
\begin{cases}
\frac{-1}{\det A} & k = n+1\\
0 & 1 < k < n+1\\
(-1)^n & k = 1
\end{cases}
\]
\end{lemma}

The proof is by integration of meromorphic differentials over the Riemann sphere, where the residues will correspond to contributions to $G_k$.  If $k = n+1$, consider

\[
\psi_j = \sum_{\mu \perp n_j - 1} \frac{(-1)^{n_j +\mu_j}}{a_j^{\mu_j + \sum_{l \neq j} \mu_l} \cdot z^{1 - \sum_{l \neq j} \mu_l}(z - a_j)} \left( \prod_{l \neq j}^m \frac{\binom{n_l+ \mu_l - 1}{\mu_l}}{(a_l - z)^{n_l + \mu_l}}  \right) \; dz.
\]

There are always poles at $z = a_1, \dots, a_m$.  There is also a pole at $z = 0$ when $\sum_{l \neq j} \mu_l = 0$, i.e. the partition is $\mu_j = n_j - 1$ and $\mu_l = 0$ for $l \neq j$.  There is never a pole at infinity.  It is immediate that 

\[
\sum_{j=1}^m \text{Res}_{a_j} \psi_{j} = G_{n+1}
\]
\noindent and
\[
\text{Res}_0 \psi_{j} = \frac{1}{\prod_{l=1}^m a_l^{n_l}} = \frac{1}{\det A},
\]

\noindent while standard power series expansion and a consideration of partition recovery shows

\[
\text{Res}_{a_j} \psi_{l} = \text{Res}_{a_j} \psi_{j}
\]

\noindent for $l \neq j$.  As the sum of residues over a closed Riemann surface is zero,

\[
0 = \sum_{j=1}^m \left[\text{Res}_0 \psi_j + \sum_{l=1}^m \text{Res}_{a_l} \psi_j \right] = \frac{m}{\det A} + m\sum_{j=1}^m \text{Res}_{a_j} \psi_j = \frac{m}{\det A} + mG_{n+1}
\]

\noindent and the $k = n+1$ case follows.


For $1 < k < n+1$, instead use 

\[
\psi_{j,k} = \sum_{i=0}^{n+1-k}\binom{n+1-k}{i} \sum_{\mu} \frac{(-1)^{n_j +\mu_j}}{a_j^{\mu_j + i + k-n + \sum_{l \neq j} \mu_l} \cdot z^{- \sum_{l \neq j} \mu_l}(z - a_j)} \left( \prod_{l \neq j}^m \frac{\binom{n_l+ \mu_l - 1}{\mu_l}}{(a_l - z)^{n_l + \mu_l}}  \right) \; dz
\]

The changed power in the denominator results in poles at $a_1, \dots, a_n$, but never at $z = 0$ or $z = \infty$.  It is again immediate that the contribution of interest is

\[
\sum_{j=1}^m \text{Res}_{a_j} \psi_{j,k} = G_k
\]

\noindent while power series expansion for this range of $k$ shows $\psi_{j,k}$ has the nice property

\[
\text{Res}_{a_j} \psi_{l,k} = \text{Res}_{a_j} \psi_{j,k}
\]

\noindent for all $l \neq j$ (this fails when $k=1$).  As the sum of residues is zero,

\[
0 = \sum_{j=1}^m \left[\sum_{l=1}^m \text{Res}_{a_l} \psi_{j,k} \right] = m\sum_{j=1}^m \text{Res}_{a_j} \psi_{j,k} = mG_k
\]

\noindent and the second part of the lemma is established.

Finally for the $k= 1$ case, use

\[
\psi_j = \sum_{i=0}^n\binom{n+1-k}{i} \sum_{\mu \perp n_j - i- 1} \frac{(-1)^{n_j +\mu_j}}{z^{1 +i - n + \mu_j}(z - a_j)} \left( \prod_{l \neq j}^m \frac{\binom{n_l+ \mu_l - 1}{\mu_l}}{(a_l - z)^{n_l + \mu_l}}  \right) \; dz.
\]

The term of interest is

\[
\sum_{j=1}^m \text{Res}_{a_j} \psi_j = G_1
\]

There is a pole at $z = \infty$ with $\text{Res}_{\infty} \psi_j = (-1)^{n+1}$, and again combinatorial manipulation shows

\[
\text{Res}_{a_j} \psi_j = \text{Res}_{a_j} \psi_l
\]

\noindent for all $l \neq j$.  We have

\[
0 = \sum_{j=1}^m \left[\sum_{l=1}^m \text{Res}_{a_l} \psi_j + \text{Res}_{\infty}\psi_j \right] = m(G_1 + (-1)^{n+1})
\]

\noindent and the final $k = 1$ case follows.

\newpage
\bibliography{refs}

\begin{thebibliography}{10}

\bibitem{arnold:1988}
{\sc Arnold, V.~I.}
\newblock {\em Geometrical Methods in the Theory of Ordinary Differential
  Equations}, 2nd~ed.
\newblock Springer, Berlin, 1988.

\bibitem{aubin:1976}
{\sc Aubin, T.}
\newblock Equations du type de {M}onge-{A}mp{\`e}re sur les vari{\'e}t{\'e}s
  {K}{\"a}hleriennes compactes.
\newblock {\em C. R. Acad. Sci. Paris 283\/} (1976), 119--121.

\bibitem{bott:1967jdg}
{\sc Bott, R.}
\newblock A residue formula for holomorphic vector-fields.
\newblock {\em J. Diff. Geom. 1\/} (1967), 311--330.

\bibitem{bott:1967mich}
{\sc Bott, R.}
\newblock Vector fields and characteristic numbers.
\newblock {\em Michigan Math. J. 14}, 2 (1967), 231--244.

\bibitem{calabi:1985}
{\sc Calabi, E.}
\newblock Extremal {K}{\"a}hler metrics ii.
\newblock In {\em Differential Geometry and Complex Analysis}, I.~Chavel and
  H.~M. Farkas, Eds. Springer, Berlin, 1985, pp.~95--114.

\bibitem{cds:2015}
{\sc Chen, X.~X., Donaldson, S., and Sun, S.}
\newblock {K}{\"a}hler-{E}instein metrics and stability.
\newblock {\em Int. Math. Res. Notices\/} (2015).

\bibitem{cherveny:2016}
{\sc Cherveny, L.}
\newblock Remarks on localizing {F}utaki-{M}orita integrals at isolated
  degenerate zeros.
\newblock {\em Diff. Geom. Appl.\/} (To Appear).

\bibitem{dervanross:2017}
{\sc Dervan, R., and Ross, J.}
\newblock K-stability for {K}{\"a}hler manifolds.
\newblock {\em Math. Res. Lett. 24}, 3 (2017), 689--739.

\bibitem{donaldson:2005}
{\sc Donaldson, S.}
\newblock Lower bounds on the {C}alabi functional.
\newblock {\em J. Diff. Geom. 70}, 3 (2005), 453--472.

\bibitem{dyrefelt:2016}
{\sc Dyrefelt, Z.~S.}
\newblock K-semistability of csc{K} manifolds with transcendental cohomology
  class.
\newblock http://arxiv.org/abs/1601.07659.

\bibitem{futaki:1983}
{\sc Futaki, A.}
\newblock An obstruction to the existence of {E}instein {K}{\"a}hler metrics.
\newblock {\em Invent. Math. 73\/} (1983), 437--443.

\bibitem{futaki:1988}
{\sc Futaki, A.}
\newblock {\em K{\"a}hler-Einstein Metrics and Integral Invariants}, vol.~1314
  of {\em Lect. Notes Math.}
\newblock Springer, 1988.

\bibitem{futakimorita:1985}
{\sc Futaki, A., and Morita, S.}
\newblock Invariant polynomials of the automorphism group of a compact complex
  manifold.
\newblock {\em J. Diff. Geom. 21\/} (1985), 135--142.

\bibitem{griffithsharris}
{\sc Griffiths, P., and Harris, J.}
\newblock {\em Principles of Algebraic Geometry}.
\newblock J Wiley and Sons, New York, 1978.

\bibitem{lishi:2015}
{\sc Li, H., and Shi, Y.}
\newblock The {F}utaki invariant on the blowup of {K}{\"a}hler surfaces.
\newblock {\em Intern. Math. Res. Notices}, 7 (2015), 1902--1923.

\bibitem{szekelyhidibook}
{\sc Sz{\'e}kelyhidi, G.}
\newblock {\em An Introduction to Extremal {K}\"ahler Metrics}, vol.~152 of
  {\em Graduate Studies in Mathematics}.
\newblock Amer. Math. Soc., 2014.

\bibitem{szekelyhidi:2015}
{\sc Sz{\'e}kelyhidi, G.}
\newblock Blowing up extremal {K}{\"a}hler manifolds {II}.
\newblock {\em Invent. Math. 200}, 3 (2015), 925--977.

\bibitem{tian:1996}
{\sc Tian, G.}
\newblock K{\"a}hler-{E}instein metrics on algebraic manifolds.
\newblock In {\em Transcendental Methods in Algebraic Geometry\/} (1996),
  F.~Catanese and C.~Ciliberto, Eds., Springer.

\bibitem{tian:2000}
{\sc Tian, G.}
\newblock {\em Canonical Metrics in {K}{\"a}hler Geomtry}.
\newblock Lectures in Mathematics ETH Z\"urich. Birkh{\"a}user, 2000.

\bibitem{yau:1978}
{\sc Yau, S.-T.}
\newblock On the {R}icci curvature of a compact {K}{\"a}hler manifold and the
  complex {M}onge-{A}mp{\`e}re equation {I}.
\newblock {\em Comm. Pure Appl. Math. 31\/} (1978), 339--411.

\end{thebibliography}
\bibliographystyle{acm}

\end{document}